\documentclass[12pt]{amsart}
\usepackage{amscd}
\usepackage{amsmath}
\usepackage{amssymb}

\usepackage[all]{xy}
\def\frk{\frak}               

\def\Phi{{\frk n}}
\def\Phi{{\frk N}}
\def\opn#1#2{\def#1{\operatorname{#2}}} 
\opn\chara{char} \opn\length{\ell} \opn\pd{pd} \opn\rk{rk}
\opn\projdim{proj\,dim} \opn\injdim{inj\,dim} \opn\rank{rank}
\opn\depth{depth} \opn\sdepth{sdepth} \opn\fdepth{fdepth}
\opn\grade{grade} \opn\height{height} \opn\embdim{emb\,dim}
\opn\codim{codim}  \opn\min{min} \opn\max{max}

\opn\Tr{Tr} \opn\bigrank{big\,rank}
\opn\superheight{superheight}\opn\lcm{lcm}
\opn\trdeg{tr\,deg}
\opn\reg{reg} \opn\lreg{lreg} \opn\ini{in} \opn\lpd{lpd}
\opn\size{size}
\opn\div{div} \opn\Div{Div} \opn\cl{cl} \opn\Cl{Cl}
\opn\Spec{Spec} \opn\Supp{Supp} \opn\supp{supp} \opn\Sing{Sing}
\opn\Ass{Ass} \opn\Min{Min}
\opn\Ann{Ann} \opn\Rad{Rad} \opn\Soc{Soc}
\opn\Im{Im} \opn\Ker{Ker} \opn\Coker{Coker} \opn\Am{Am}
\opn\Hom{Hom} \opn\Tor{Tor} \opn\Ext{Ext} \opn\End{End}
\opn\Aut{Aut} \opn\id{id}  \opn\deg{deg}

\opn\nat{nat}
\opn\pff{pf}
\opn\Pf{Pf} \opn\GL{GL} \opn\SL{SL} \opn\mod{mod} \opn\ord{ord}
\opn\Gin{Gin} \opn\Hilb{Hilb}
\opn\aff{aff} \opn\con{conv} \opn\relint{relint} \opn\st{st}
\opn\lk{lk} \opn\cn{cn} \opn\core{core} \opn\vol{vol}
\opn\link{link} \opn\star{star}
\opn\gr{gr}

\def\pot#1#2{#1[\kern-0.28ex[#2]\kern-0.28ex]}

\opn\dirlim{\underrightarrow{\lim}}
\opn\inivlim{\underleftarrow{\lim}}
\let\to=\rightarrow

\def\Implies{\ifmmode\Longrightarrow \else
        \unskip${}\Longrightarrow{}$\ignorespaces\fi}
\def\implies{\ifmmode\Rightarrow \else
        \unskip${}\Rightarrow{}$\ignorespaces\fi}
\def\iff{\ifmmode\Longleftrightarrow \else
        \unskip${}\Longleftrightarrow{}$\ignorespaces\fi}

\let\:=\colon
\newtheorem{Theorem}{Theorem}[section]
\newtheorem{Lemma}[Theorem]{Lemma}

\newtheorem{Proposition}[Theorem]{Proposition}

\newtheorem{Example}[Theorem]{Example}

\let\epsilon\varepsilon
\let\phi=\varphi
\let\kappa=\varkappa
\def\qed{\ifhmode\textqed\fi
      \ifmmode\ifinner\quad\qedsymbol\else\dispqed\fi\fi}
\def\textqed{\unskip\nobreak\penalty50
       \hskip2em\hbox{}\nobreak\hfil\qedsymbol
       \parfillskip=0pt \finalhyphendemerits=0}
\def\dispqed{\rlap{\qquad\qedsymbol}}

\opn\dis{dis}
\def\pnt{{\raise0.5mm\hbox{\large\bf.}}}

\opn\Lex{Lex}

\begin{document}
\title{\bf Depth of some square free monomial  ideals}

\author{ Dorin Popescu and Andrei Zarojanu}

\thanks{The  support from  grant ID-PCE-2011-1023 of Romanian Ministry of Education, Research and Innovation is gratefully acknowledged.}

\address{Dorin Popescu,  "Simion Stoilow" Institute of Mathematics of Romanian Academy, Research unit 5,
 P.O.Box 1-764, Bucharest 014700, Romania}
\email{dorin.popescu@imar.ro}
\address{Andrei Zarojanu,  Faculty of Mathematics and Computer Sciences, University
of Bucharest, Str. Academiei 14, Bucharest, Romania}
\email{andrei\_zarojanu@yahoo.com}

\maketitle
\maketitle
\begin{abstract} Let $I\supsetneq J$ be  two  square free monomial ideals of a polynomial algebra over a field generated in degree $\geq 1$, resp. $\geq 2$ .  Almost always when $I$ contains precisely one variable, the other generators having degrees $\geq 2$, if  the Stanley depth of $I/J$ is $\leq 2$ then the usual depth of $I/J$ is $\leq 2$ too, that is the Stanley Conjecture holds in these cases.

 \noindent
  {\it Key words } : Monomial Ideals,  Depth, Stanley depth.\\
 {\it 2000 Mathematics Subject Classification: Primary 13C15, Secondary 13F20, 13F55,
13P10.}
\end{abstract}

\section*{Introduction}

 Let $K$ be a field, $S=K[x_1,\ldots,x_n]$ be the polynomial algebra in $n$ variables over $K$ and   $I\supsetneq J$  two  square free monomial ideals of $S$. We assume that $I$, $J$ are generated by square free monomials of degrees $\geq d$, resp. $\geq d+1$ for some $d\in {\bf N}$.  Then $\depth_S I/J\geq d$ (see  \cite[Proposition 3.1]{HVZ}, \cite[Lemma 1.1]{P}). Upper bounds of $\depth_S I/J$ are given by numerical conditions in \cite{P2}, \cite[Theorem 2.2]{P}, \cite[Theorem 1.3]{P1} and \cite[Theorem 2.4]{Sh}. An important tool in the proofs is the Koszul homology, except in the last quoted paper, where the results are stronger, but the proofs are extremely short relying completely on some results concerning the Hilbert depth, which proves there  to be a very strong tool (see \cite{BKU}, \cite{U} and \cite{IM}). These results are inspired by the so called the Stanley Conjecture, which we explain below.

   Let $P_{I\setminus J}$  be the poset of all square free monomials of $I\setminus J$ (a finite set) with the order given by the divisibility. Let ${\mathcal P}$ be a partition of  $P_{I\setminus J}$ in intervals $[u,v]=\{w\in  P_{I\setminus J}: u|w, w|v\}$, let us say   $P_{I\setminus J}=\cup_i [u_i,v_i]$, the union being disjoint.
Define $\sdepth {\mathcal P}=\min_i\deg v_i$ and  the so called {\em Stanley depth} of $I/J$ given by $\sdepth_SI/J=\max_{\mathcal P} \sdepth {\mathcal P}$, where ${\mathcal P}$ runs in the set of all partitions of $P_{I\setminus J}$ (see  \cite{HVZ}, \cite{S}).
The Stanley depth is not easy to handle, see \cite{HVZ}, \cite{R}, \cite{Is}, \cite{HPV} for some of its properties.

 Stanley's Conjecture says that  $\sdepth_S I/J\geq \depth_S I/J$. Thus the Stanley depth of $I/J$ is a natural combinatorial upper bound of
 $ \depth_S I/J$ and the above results give  numerical conditions to imply upper bounds of $ \sdepth_S I/J$. When $J=0$ the Stanley Conjecture holds either
when $n\leq 5$ by \cite{P3}, or when $I$ is an intersection of four monomial  prime ideals by \cite{AP}, \cite{P4}, or when $I$ is an intersection of three primary ideals by \cite{z}, or when $I$ is an almost complete intersection by \cite{Cim}.

 Let $r$ be the number of the square free monomials of degree $d$ of $I$ and $B$ (resp. $C$) be the set of the square free monomials of degrees $d+1$  (resp. $d+2$) of $I\setminus J$.  Set $s=|B|$, $q=|C|$. If either $s>r+q$, or $r>q$, or $s<2r$ then $\sdepth_SI/J\leq d+1$ and if the Stanley Conjecture holds then any of these numerical conditions would imply $\depth_SI/J\leq d+1$. In particular this was proved  directly in \cite{P1} and \cite{Sh}.

 Now suppose that $I$ is generated by one variable and some square free monomials of degrees $\geq 2$. It is the purpose of our paper to show that  almost always  if  $\sdepth_SI/J\leq 2$ then  $\depth_SI/J\leq 2$ (see our Theorem \ref{p1}). It is known already that  $\sdepth_SI/J\leq 1$ implies  $\depth_SI/J\leq 1$ (see \cite[Theorem 4.3]{P}) and so our Theorem
\ref{p1} could be seen as a new step (small but difficult) in the study of Stanley's Conjecture.

\section{Stanley depth of some square free monomial ideals}

Let $I\supsetneq J$ be  two  square free monomial ideals of $S$. We assume that $I$, $J$ are generated by square free monomials of degrees $\geq d$, resp. $\geq d+1$ for some $d\in {\bf N}$.
As above $B$ (resp. $C$) denotes the set of the square free monomials of degrees $d+1$  (resp. $d+2$) of $I\setminus J$.
\begin{Lemma} \label{r} Suppose that $d=1$, $I=(x_1,\ldots,x_r)$ for some  $1\leq r< n$ and   $J\subset I$  be a square free monomial ideal generated in degree $\geq 2$. Let $B$ be the set of all square free monomials of degrees $2$ from $I\setminus J$. Suppose that
 $\depth_S I/(J+((x_j)\cap B))=1$ for some $r<j\leq n$. Then  $\depth_S I/J \leq 2$.
\end{Lemma}
\begin{proof} Since $I/(J+((x_j)\cap B))$ has a square free, multigraded free resolution we see that only the  components of square free degrees of
$$\Tor_{n-1}^S(K,I/(J+(x_j)\cap B)))\cong H_{n-1}(x;I/(J+(x_j)\cap B))$$ are nonzero. Thus  we may find
$z=\sum_{i=1}^r y_ix_ie_{[n]\setminus \{\i\}}\in K_{n-1}(x;I/(J+(x_j)\cap B)$, $y_i\in K$ inducing a nonzero element in $H_{n-1}(x;I/(J+(x_j)\cap B)$. Here we denoted $e_{\tau}=\wedge_{j\in \tau}\ e_j$ for a subset $\tau\subset [n]$. Then we see that
$$z'=\sum_{i=1}^r y_ix_ie_{[n]\setminus \{i,j\}}\in K_{n-2}(x;I/J)$$ induces a nonzero element in $H_{n-2}(x;I/J)$. Thus $\depth_SI/J\leq 2$ (see \cite[Theorem 1.6.17]{BH}).
\end{proof}
\begin{Example} {\em Let $n=4$, $r=2$, $d=1$, $I=(x_1,x_2)$, $J=(x_1x_2)$, $B=\{x_1x_3,x_1x_4, x_2x_3,x_2x_4\}$. Then
$F=I/(J + (x_1) \cap B ) \cong (x_1,x_2) / ((x_1) \cap (x_2,x_3,x_4))$ has sdepth and depth $=1$, but $\depth_S I/J=3$. Thus the statement of the above lemma can be false if $j<r$. More precisely, $\depth_S F=1$ because $z=x_1e_{234}$ induces a nonzero element in $H_{3}(x;F)$ but $e_1$ is not present in $e_{234}$.}
\end{Example}
\begin{Proposition} \label{p} Suppose that $I \subset S$ is generated by $\{x_1,\ldots,x_r\}$ for some  $1\leq r\leq n$ and some square free monomials of degrees $\geq 2$,  and $x_ix_tx_k\in J$ for all $i\in [r]$ and $r<t<k\leq n$. Then $\depth_S I/J\leq 2$.
\end{Proposition}
\begin{proof} First suppose that $I=(x_1,...,x_r)$.
If there exists $j>r$ such that $\depth_S I/(J+(x_j)\cap B)=1$ then we may apply the above lemma.
Thus we may suppose that $\depth_S I/(J+(x_j)\cap B)\geq 2$ for all $j>r$. Assume that $\depth_SI/J>2$. By decreasing induction on $r<t\leq n$ we show that  $\depth_S I/(J+(x_t,\ldots,x_n))\cap B)\geq 2$. We assume that $t<n$ and $\depth_S I/(J+(x_{t+1},\ldots,x_n))\cap B)\geq 2$, $\depth_S I/(J+(x_{t},\ldots,x_n))\cap B)=1$. Set $L=(J+(x_t)\cap B)\cap (J+(x_{t+1},\ldots,x_n)\cap B)$.
In the following exact sequence
$$0\to I/L\to I/(J+(x_t)\cap B)\oplus I/(J+(x_{t+1},\ldots,x_n)\cap B)\to I/(J+(x_{t},\ldots,x_n)\cap B)\to 0$$
the last term has the depth $1$ and the middle the depth $\geq 2$. By the Depth Lemma we get $\depth_S I/L=2$.

Remains to show that $\depth_SI/J=\depth_SI/L$. Note that there exist no $c\in C$ multiple of $x_tx_j$ for some $r<t<j\leq n$ by our hypothesis. Thus  $L=J$. Then it follows $\depth_SI/J=2$ which contradicts our assumption.
The induction ends for $t=r+1$ and we get  $\depth_S I/(J+(x_{r+1},\ldots,x_n)\cap B)=2$; but this is not possible (see for example \cite[Lemma 1.8]{P}).

Now suppose that $I=U+V$, where $U=(x_1,...,x_r)$ and $V$ is generated by some square free monomials of degrees $\geq 2$. In the following exact sequence
$$0\to U/(U\cap J) \to I/J\to I/(U+J)\to 0$$
the first term has depth $\leq 2$ from above and the last term is isomorphic with $V/(V\cap (U+J))$ and has depth $\geq 2$ by \cite[Lemma 1.1]{P}. So by the Depth Lemma it follows that $\depth I/J\leq 2$.
\end{proof}

\begin{Example}{\em Let $n=4$, $I=(x_1,x_2,x_3)$, $J=(x_1x_3)$. Clearly, $B_1=\emptyset$,\\
 $B=\{x_1x_2,x_1x_4,x_2x_3,x_2x_4,x_3x_4\}$  and $C=\{x_1x_2x_4,x_2x_3x_4\}$. We have $s=5$, $r=3$, $q=2$ and so $s=r+q$. Note that each $c\in C$ is a multiple of a monomial of the form $x_ix_j$ for some $1\leq i<j\leq 3$ and so $\depth_S I/J\leq 2$ by the above proposition. On the other hand, it is easy to see that
 $z=x_1e_2\wedge e_3-x_2e_1\wedge e_3+x_3e_1\wedge e_2$ induces a nonzero element in $H_2(x;I/J)$ and so again $\depth_S I/J\leq 2$.}
\end{Example}

\begin{Lemma}\label{za}
If a monomial $u$ of degree $k$ from $I \setminus J$ has all multiples
 of degrees $k+1$ in $J$ then $\depth I/J \leq k$.
\end{Lemma}
\begin{proof}
Renumbering the variables $x$ we may suppose that $u=x_1\cdots x_k$. Then we see that $u(x_{k+1},...,x_n) = 0$ so $\Ann_S u = (x_{k+1},...,x_n) \in \Ass_S I/J.$ Thus $\depth I/J \leq k$.
\end{proof}
\begin{Lemma} \label{e} Suppose that $J\subset I$  are square free monomial ideals generated in degree $\geq d+1$, respectively $\geq d$ and let $V$ be an ideal    generated by $e$  square free monomials of degrees $\geq d+2$, which are not in $I$. Then $\sdepth_S(I+V)/J\leq d+1$ (resp. $\depth_S(I+V)/J\leq d+1$) implies that $\sdepth_SI/J\leq d+1$ (resp.  $\depth_SI/J\leq d+1$).  For the depth the converse is also true.
\end{Lemma}
\begin{proof}
By induction on $e$, we may consider only the case $e=1$, that is  $V=\{v\}$.
In the following  exact sequence
$$0\to I/J\to (I+V)/J\to (I+V)/(I+J)\to 0$$
the last term is isomorphic with $(v)/((v) \cap (I+J))$ and has depth and sdepth $\geq d+2$. Then the first term has sdepth $\leq d+1$ by \cite[Lemma 2.2]{R}  and depth $\leq d+1$ by the Depth Lemma.
\end{proof}

\begin{Lemma} \label{p+2} Suppose that $I \subset S$ is generated by $x_1,\ldots,x_r$ and a nonempty set $E$ of square free monomials of degrees $2$ in the variables $x_{r+1},\ldots,x_n$, and $\sdepth_S I/J=2$.
Let $x_1x_t\in B$ for some $t$, $r<t\leq n$, $I'=(x_2,\ldots,x_r)+(B\setminus \{x_1x_t\})$, $J'=J\cap I'$ and $\mathcal P$ a partition of $I'/J'$ with sdepth $3$.
Assume that
any square free monomial $u\in S$ of degree $2$, which is not in $I$, satisfies $x_1u\in J$.
Then
\begin{enumerate}
\item For any $a\in (B\setminus (x_2,\ldots,x_r,x_1x_t))\cap (x_t)$ with $x_1a\not\in J$  the interval $[a,x_1a]$ is in $\mathcal P$.
\item If $c=x_tx_ix_j\not\in J$, $r<i<j\leq n$, $i,j\not =t$  and $x_1x_tx_i, x_1x_tx_j\not\in J$ then $b=c/x_t\in B$ and if moreover $x_1b\not\in J$ then $c$ is not present in an interval $[a,c]$, $a\in B$ of  $\mathcal P$.
 \end{enumerate}
    \end{Lemma}
\begin{proof} Let   $a=x_tx_{\nu}$ be a monomial   of $ B \setminus (x_2,\ldots,x_r,x_1x_t))$ with satisfies $x_1a\not\in J$. Suppose that  the interval $[a,x_1a]$ is  not in $\mathcal P$. Then there exists in $\mathcal P$ an interval $[a,c]$ with $c\in C$.    Thus $x_1x_{\nu}$ is in $B$ and so in  $\mathcal P$ there exists  an interval $ [x_1x_{\nu},c']$, $c'\in C $,. We replace   the   interval  $[x_1x_{\nu},c']$ by $ [x_1,x_1a] $ to get a  partition of $ I/J$ with sdepth $\geq 3$. However, such  partition of $ I/J$
  is not possible because $\sdepth_SI/J=2$. Thus    the interval $ [a,x_1a]$
is   in  $\mathcal P$.

Now, let $c$ be as in $(2)$. We will show that  $b= c/x_t\in B$.  Indeed, if $ b\not\in B$ then $b\not\in (x_1,\ldots,x_r)$ because otherwise $b\in J$, which is false. Thus $c$ can enter only   in an  interval $ [a,c]$ for let us say $a=x_tx_i$.  But this interval is not  in $\mathcal P$ because $ a$ belongs to the  interval $[a,x_1a]$. Contradiction!  Thus $c$ does not appear in the intervals of $\mathcal P$.  Replacing $[a, x_1a]$ with $ [a,c]$ in $ \mathcal P$ we get another partition of $I'/J'$ with sdepth $3$,  where the interval  $[a,x_1a]$ is not present, contradicting $(1)$.

 Moreover suppose that $x_1b\not \in J$.  By $(1)$,  $c$ can appear only in  the interval $[b,c]$  because  we have already the intervals $[x_tx_i,x_1x_tx_i]$, $[x_tx_j,x_1x_tx_j]$ in  $\mathcal P$. Then we cannot have an interval $[b,x_1b]$ in  $\mathcal P$ and so $x_1b$ could appear in the interval, let us say $[x_1x_i,x_1b]$. Certainly, it is possible that $x_1b$ will not appear at all in an interval of $\mathcal P$, but we may modify $\mathcal P$ to get this. Replace in  $\mathcal P$ the intervals $[x_1x_i,x_1b]$, $[b,c]$, $[x_tx_i,x_1x_tx_i]$ by the intervals $[b,x_1b]$, $[x_tx_i,c]$, $[x_1x_i,x_1x_tx_i]$ and we get another partition of $I'/J'$ with sdepth $3$ but without  the interval $[x_tx_i,x_1x_tx_i]$, contradicting again (1).
\end{proof}

\begin{Lemma} \label{p-} Suppose that $I \subset S$ is generated by $x_1$ and a nonempty set $E$ of square free monomials of degrees $2$ in  $x_2,\ldots,x_n$ and $\sdepth_S I/J=2$.
Assume that $x_1a\not \in J$ for all $a\in E$ and
any square free monomial $u\in S$ of degree $2$, which is not in $I$, satisfies $x_1u\in J$.
Then $\depth_S I/J\leq 2$.
\end{Lemma}
\begin{proof} Let $1<t\leq n$ be such that $x_1x_t\in B$. We may assume that $ a_1,\ldots,a_k$, are all monomials  of $ (E\cap (x_t))\setminus \{x_1x_t\}$. Set $ I_t=(B \setminus \{x_1x_t\}) $ and $ J_t=J\cap I_t$. In the  exact sequence
$$0\to I_t/J_t\to I/J\to I/J+I_t\to 0$$
the last term has depth $\geq 2$ because it is isomorphic with $ (x_1)/(x_1)\cap (J+I_t)$ and $x_1x_t\not\in  J+I_t$. If $\sdepth_S I_t/J_t\leq 2$ then we get  $\depth_S I_t/J_t \leq 2$ by \cite[Theorem 4.3]{P}. Applying the Depth Lemma we get  $\depth_S I/J\leq 2$.

 Thus we may assume that $\sdepth_S I_t/J_t\geq 3$ for all $1<t\leq n$ such that $x_1x_t\in B$. Let ${\mathcal P}={\mathcal P_t}$ be a partition of $I_t/J_t$ with sdepth $=3$. By the above lemma    the intervals $ [a_j,x_1a_j]$,  $1\leq j\leq k$ are in  $ \mathcal P$.

 Suppose that $c=x_ix_jx_t\in C$, $i,j,t>1$  and $x_{j}x_t,x_{i}x_t\in E$.  Then $a=x_ix_j\in E$ by the above lemma.  By our hypothesis we have $x_1a, x_1x_{j}x_t,x_1x_{i}x_t\in C$.   Thus  $c$ cannot appear in an interval of $\mathcal P$ using again the above lemma.

 For $b=x_1x_i\in B$,  $\mathcal P$ must contain  some  intervals   of the form $[x_1x_i,x_1a'_i]$ for some $a'_i\in E$. Certainly $ a'_i\not\in (x_t)$ because we saw that all $ a_j$, $ 1\leq j\leq k$ enter already  in the intervals $ [a_j,x_1a_j]$.
 Then these $ a'_i$ enter   in some intervals $[a'_i,c'_{i}]$ with $ c'_{i}\in (C\setminus (x_1))$.
   If  $c'_{i}\in (a_j)$ for some $ a_j$, $1\leq j\leq k$ then the third divisor of $c'_{i}$ of degrees $2$ is in $B$ too, and as above $c'_i$ cannot appear in an interval of ${\mathcal P}$. Contradiction! Thus $ c'_{i}\in (C\setminus (x_1,a_1,\ldots,a_k))$.

Let $ I'=(x_1x_t,a_1,\ldots,a_k)$, $J'=J\cap I'$. We have seen that $c'_{i}\not\in I'$.  In the exact sequence
$$0\to I'/J'\to I/J\to I/J+I'\to 0$$
we show that the last term has sdepth $\geq 3$. Let $a'_{i}=x_{i}x_{\nu_i}\in B$ for some $1<\nu_i\leq n$.  We may suppose that $t>2$, $x_1x_2\in B$ and we see that the intervals $[x_1,x_1a'_{2}]$,
$ [x_1x_i,x_1a'_i]$, $i>2$, $i\not =\nu_2$, $ [a'_i,c'_{i}]$   induce with the help of $\mathcal P$ a partition of  $I/J+ I'$  with sdepth $3$. Indeed, the only possible problem is that in $\mathcal P$ could appear some intervals of type $[a,ax_t]$ for some $a\in (E\setminus (x_t))$, $c=ax_t$ being the least common multiple of two $(a_j)$. But this is not possible as we saw above.
By  \cite[Lemma 2.2]{R} we get
$\sdepth_S I'/J'\leq 2$ and so $ \depth_S I'/J'\leq 2$ by \cite[Theorem 4.3]{P}.  Applying the Depth Lemma we get as  $\depth_S I/J\leq 2$.

\end{proof}

\begin{Proposition} \label{p0} Suppose that $I \subset S$ is generated by $x_1$ and a nonempty set $E$ of square free monomials of degrees $2$ in  $x_2,\ldots,x_n$ and $\sdepth_S I/J=2$.  Let  $E'=\{a\in E: x_1a\in C\}$ and $E''=E\setminus E'$.
Assume that any square free monomial $u\in S$ of degree $2$, which is not in $I$, satisfies $x_1u\in J$ and one of the following conditions hold:
\begin{enumerate}
\item{} $|E''|\leq |C\setminus (x_1,E')|$
\item{}  $|E''|> |C\setminus (x_1,E')|$ and $|B|\not =|C|+1$.
 \end{enumerate}
Then $\depth_S I/J\leq 2$.
\end{Proposition}
\begin{proof}   If  $E''=\emptyset$ then  we apply the above lemma. Apply induction on $|E''|$. If $E'=\emptyset$ then  $C\cap (x_1)=\emptyset$ and the conclusion follows from Lemma \ref{za}. Let  $E''=\{a_1,\ldots,a_k\}$, $k> 0$. We claim that we may reduce our problem to the case when  $(C\setminus (x_1))\subset (E'')$. Indeed, otherwise let $c\in (C\setminus (x_1,E''))$. Then there exists $b\in E'$ such that $c\in (b)$. Choose $t$, $1<t\leq n$ such that $x_t|b$. Then $x_1x_t$ divides $x_1b\in C$ and so it is in $B$. Set $I'=(B\setminus \{x_1x_t\})$, $J'=J\cap I'$.
  In the following exact sequence
$$0\to I'/J'\to I/J\to I/(I'+J)\to 0$$                                                                                                                                                                                             the last term is isomorphic with $(x_1)/(x_1)\cap (I'+J)$ and has depth $\geq 2$ because $x_1x_t\not \in (I'+J)$.
  If  $\sdepth_S I'/J' \leq 2$ then by \cite[Theorem 4.3]{P}  we get $\depth_S I'/J'\leq 2$ and using the Depth Lemma it   follows $\depth_SI/J\leq 2$.

Thus we may suppose that $\sdepth_SI'/J'\geq 3$ and let ${\mathcal P}={\mathcal P}_t$ be a partition of $I'/J'$ with sdepth $3$. By Lemma $\ref{p+2}$ (see also the above lemma), $\mathcal P$ may contain some disjoint intervals    $ [x_1x_i,x_1b'_i]$, $[b'_i,c'_i]$,  for some $b'_i\in E'$, $c'_i\in C\setminus (x_1)$, $i\not =1,t$ with
$x_1x_i\in B$, $[b',x_1b']$ for $b'\in E'\setminus \{\{b'_i\}\}$ and $[a_j,c_j]$, $j\in [k]$, $c_j\in C$. As in the proof of the above lemma we have $b'_i\not\in (x_t)$. Thus  the above $b$ is not one of $b'_i$ and enters in $\mathcal P$ in the interval $[b,x_1b]$. Note that $c$ is not among   $\{ \{c_j\}\}$ because is not in $(E'')$. On the other hand, if $c=c'_i$ then should be divisible by $b$ and $b'_i$, both being from $E'$. Then by Lemma $\ref{p+2}$ applied for a  $t'$ given by the only one common variable $x_{t'}$ of  $b$, $b'_i$, the third divisor $u=c/x_{t'}$ of degree $2$ of $c$ is in $E$,  and   $x_1u\in J$ because $c$
can enter in an interval $[u,c]$ of a partition  ${\mathcal P}_{t'}$. Thus $u\in E''$ and so $c\in (E'')$, which is false.    Then we may replace the  interval $[b,x_1b]$ by $[b,c]$, which is again false because all intervals $[b',x_1b']$, $b'\in (E')\cap (x_t)$ should be present in $\mathcal P$ by Lemma $\ref{p+2}$. This proves our claim. Also note that $|C\setminus (x_1)| \geq
|B\cap (x_1)|-1+k$.

Then we may assume that
$(C\setminus (x_1))\subset  (E'')$.
We may suppose that $c_i\in (E')$ if and only if $p< i\leq k$ for some $0\leq p\leq k$. Moreover, we will arrange to have as many as possible  $c_j$ outside  $(E')$.  If $c'\in (C\setminus (x_1))$ is a multiple of let say $a_{p+1}$, but $c'\not\in (E')$. We may replace in the above intervals $c_{p+1}$ by $c'$, the effect being  the increasing of $p$. Thus after such procedure we may suppose that either $p=k$, or there exist no $c$ in $(C\setminus (x_1,c_{1},\ldots,c_p))\cap (a_{p+1},\ldots,a_k)$ which is not in $(E')$.

If $p=k$ then set $I''=(x_1,E')$, $J''=I''\cap J$ and see that in the
exact sequence
$$0\to I''/J''\to I/J\to I/(I''+J)\to 0$$                                                                                                                                                                                             the last term is isomorphic with $(E'')/(E'')\cap (I''+J)$ and has sdepth $ 3$ because  the intervals $[a_j,c_j]$, $j\in [k]$ gives a partition with sdepth $3$.
  Then $\sdepth_S I''/J'' \leq 2$  by \cite[Proposition 2.2]{R} and  we get $\depth_S I''/J''\leq 2$ by Lemma \ref{p-}. Using the Depth Lemma it   follows $\depth_SI/J\leq 2$.

 Next suppose that $ p<k$. Then $(C\setminus (x_1,c_{1},\ldots,c_p))\cap (a_{p+1},\ldots,a_k)\subset (E')$.
 We may  choose $c_1,\ldots,c_p$ from the beginning (it is possible to make such changes in $\mathcal P$) such that
  $e=|\{i:c_i\not \in (a_{p+1},\ldots, a_k)\}|$ is maxim possible and renumbering $a_j$, $j\leq p$ we may suppose that $c_i\not \in (a_{p+1},\ldots,a _k)$ if and only if $i\in [e]$ for some $0\leq e\leq p$.

    Suppose that there exists  $c\in C\setminus (x_1,c_1,\ldots,c_p)$ such that   $c\not\in E'$. Then $c$ is not in $(a_{p+1},\ldots,a_k)$ and necessary $c\in (a_1,\ldots,a_p)$. Assume that $c\in (a_i)$ for some $i\in [p]$. If $i>e$ then $c_i\in (a_{p+1},\ldots,a _k)$, let us say $c_i\in (a_j)$ for some $j>p$ and we may change $c_j$ by $c_i$ and  replace $c_i$ by $c$ increasing $p$ because $c_i\not\in E'$. This is not possible since $p$ was maxim given. Thus $i\leq e$ and so $e>0$. If $c_i\in (a_{e+1},\ldots,a_p)$, let us say $c_i\in (a_p)$ then we may replace $c_p$ by $c_i$ and $c_i$ by $c$ increasing $e$ which is also not possible. Thus $c_i\not \in (a_{e+1},\ldots,a_p)$.

    Then set $I_e=(x_1,B\setminus \{a_1,\ldots,a_e\})$, $J_e=I_e\cap J$. In the exact sequence
 $$0\to I_e/J_e\to I/J\to I/(I_e+J)\to 0$$
 the last term has sdepth $3$ because we may write there the intervals $[a_i,c_i], i\in [e]$ since $c_i\not \in I_e$. By \cite[Proposition 2.2]{R}
  it follows that $\sdepth_S I_e/J_e\leq 2$ and so  $\depth_S I_e/J_e\leq 2$ by induction hypothesis on $|E''|$. Using the Depth Lemma it follows $\depth_S I/J \leq 2$.

  Now suppose that there exist no such $c$, that is $C\setminus (x_1,E')=\{c_1,\ldots,c_p\}$. Thus $p=|C\setminus (x_1,E')|$ and so we end the case when the condition (1) holds.  Now suppose that the condition (2) holds, in particular   $k>p$ and $s=|B|\not=1+q$  for $q=|C|$. If  $s>1+q$ then we end with \cite{P1}. Suppose that $s<1+q$. Then there exists a $c\in C$ which does not appear in an interval $[b,c]$ for some $b\in (B\setminus \{x_1x_t\})$. Note that $c$
 cannot be a $c_j$ for $j\in [p]$ and so $c\in (E')$, let us say $c\in (a)$ for some  $a\in E'$. Let $j$ be such that $x_j|a$. We have $x_1x_j\in B$ and there exists as above a partition ${\mathcal P}_j$ with sdepth $3$. Let $I_a=(B\setminus \{a\})$, $J_a=I_a\cap J$. We see that ${\mathcal P}_j$ induces a partition
 ${\mathcal P}_a$ of $I_a/J_a$ with sdepth $3$ replacing the interval $[a,x_1a]$ from ${\mathcal P}_j$ with $[x_1x_j,x_1a]$.

  In ${\mathcal P}_a$ there is an interval $[x_1x_t,x_1a''_1]$ for some $a''_1=x_tx_i\in E'$. We have $a''_1\not=a'$ because otherwise we may change in ${\mathcal P}_t$
  the interval $[a''_1,x_1a''_1]$ by $[a''_1,c]$, which is false. Then there is in ${\mathcal P}_a$ an interval $[a''_1,c''_1]$. If $c''_1$ is not a $c_b$ as above then we may replace in  ${\mathcal P}_t$ the interval  $[a''_1,x_1a''_1]$ by $[a''_1,c''_1]$, which is again false. Thus $c''_1=c_{b_1}$ for some $b_1\in (B\setminus \{x_1x_t\})$.  If $b_1=a$ we may replace in  ${\mathcal P}_t$ the intervals $[a''_1,x_1a''_1]$, $[b_1,c''_1]$ by $[a''_1,c''_1]$, $[b_1,c]$, which is false.  Then there is in ${\mathcal P}_a$ an interval $[b_1,c''_2]$. By recurrence we find in ${\mathcal P}_a$ the intervals $[x_1x_t,x_1a''_1]$, $[a''_1,c''_1]$, $[a''_2,c''_2],\ldots $ which define a partition ${\mathcal P}_a$, where $c$ is not present in an interval $[b,c]$, $b\in (B\setminus\{a\})$.  Adding the interval $[a,c]$ to ${\mathcal P}_a$ we get a partition ${\mathcal P}'$  with sdepth $3$ of $I_B/J_B$, where $I_B=(B)$, $J_B=I_B\cap J$. But then we replace in  ${\mathcal P}'$ the intervals $[x_1x_t,x_1a''_1]$, $[x_1x_i,x_1a''_1]$ by $[x_1,x_1a''_1]$ and we get a partition of $I/J$ with sdepth $3$. Contradiction!
  \end{proof}

\begin{Theorem} \label{p1} Suppose that $I \subset S$ is generated by $x_1$ and a nonempty set $E$ of square free monomials of degrees $2$ in  $x_2,\ldots,x_n$ and $\sdepth_S I/J=2$.  Let  $E'=\{a\in E: x_1a\in C\}$ and $E''=E\setminus E'$. Assume that one of the following conditions holds:
\begin{enumerate}
\item{} $|E''|\leq |C\setminus (x_1,E')|$
\item{}  $|E''|> |C\setminus (x_1,E')|$ and $|B|\not =|C|+1$.
\end{enumerate}
Then $\depth_S I/J \leq 2$.
\end{Theorem}
\begin{proof}  We may assume $n>2$ and  there exists   $c=x_1x_{n-1}x_n\not \in J$ after renumbering the variables $x$,  otherwise we  apply Proposition $\ref{p}$.
Then $z=x_{n-1}x_n\not \in J$.

 First suppose that we may find $c$ with $z\not \in I$.  Set $I'=(B\setminus \{x_1x_{n-1},x_1x_n\})$ and $J'=I'\cap J$. Then necessary $B\supsetneq \{x_1x_{n-1},x_1x_n\}$ and so $I'\not =J'$ because otherwise $\sdepth_S I/J=3$. Note that no $b$ dividing $c$ belongs to $I'$ and so $c\not \in (J+I')$. In the following exact sequence                                                                                                                                                                                                       $$0\to I'/J'\to I/J\to I/(I'+J)\to 0$$                                                                                                                                                                                             the last term  has sdepth $\geq 3$ since  $[x_1,c]$ is the whole poset  of $(x_1)/(x_1)\cap (I'+J)$  except some  monomials of degrees $\geq 3$. It has also depth $\geq 3$ because $x_{n-1}x_n\not \in ((J+I'):x_1)$.
  The first term has sdepth $\leq \sdepth_S I/J=2$ by \cite[Lemma 2.2]{R} and so it has depth $\leq 2$ by \cite[Theorem 4.3]{P}. It   follows $\depth_SI/J\leq 2$.

  Next suppose that there exist no such $c$, that is any square free monomial  $u\in S$ of degree $2$, which is not in $I$ satisfies $x_1u\in J$. We may assume that $C\subset (x_1,B)$ by Lemma $\ref{e}$. Now it  is enough to apply Proposition $\ref{p0}$.
\end{proof}

\begin{Example} {\em Let $n=3$, $r=1$, $I=(x_1,x_2x_3)$, $J=0$. We have $c=x_1x_2x_3\not \in J$ and $x_2x_3\in I$. Note also that $\sdepth_S I=\depth_S I=2.$}
\end{Example}


\begin{thebibliography}{99}

\bibitem{BH} W.\ Bruns and J. Herzog, {\em Cohen-Macaulay rings}, Revised edition. Cambridge University Press (1998).


\bibitem{BKU} W.\ Bruns, C.\ Krattenthaler, J.\ Uliczka, {\em Stanley decompositions and Hilbert depth in the Koszul complex}, J. Commutative Alg., {\bf 2} (2010), 327-357.
\bibitem{Cim} M.\ Cimpoeas, {\em The Stanley conjecture on monomial almost complete intersection ideals}, Bull. Math. Soc. Sci. Math. Roumanie, {\bf 55(103)}, (2012), 35-39.
\bibitem{HVZ} J.\ Herzog, M.\ Vladoiu, X.\ Zheng, {\em How to compute the Stanley depth of a monomial ideal,}  J.  Algebra, {\bf 322} (2009), 3151-3169.
\bibitem{HPV} J.\ Herzog, D.\ Popescu, M.\ Vladoiu, {\em Stanley depth and size  of a monomial ideal}, Proc.  Amer. Math. Soc., {\bf 140} (2012), 493-504, arXiv:AC/1011.6462v1.
\bibitem{IM} B.\ Ichim, J. J.\ Moyano-Fernandez, {\em How to compute the multigraded Hilbert depth of a module}, (2012), arXiv:AC/1209.0084.
\bibitem{Is} M.\ Ishaq, {\em Upper bounds for the Stanley depth}, Comm. Algebra {\bf 40}(2012), no. 1, 87–97.
\bibitem{AP} A.\ Popescu, {\em Special Stanley Decompositions}, Bull. Math. Soc. Sc. Math. Roumanie, {\bf 53(101)}, no 4 (2010),363-372, arXiv:AC/1008.3680.
\bibitem{P3} D.\ Popescu,  {\em An inequality between depth and Stanley depth}, Bull. Math. Soc. Sc. Math. Roumanie {\bf 52(100)}, (2009), 377-382, arXiv:AC/0905.4597v2.
\bibitem{P4} D.\ Popescu, {\em Stanley conjecture on intersections of four monomial prime ideals}, to appear in Communications in Alg.,arXiv:AC/1009.5646.

\bibitem{P2} D.\ Popescu, {\em Depth and minimal number of generators of square free monomial ideals}, An. St. Univ. Ovidius, Constanta, {\bf 19 (2)}, (2011), 187-194.

\bibitem{P} D.\ Popescu, {\em Depth of  factors of square free  monomial  ideals}, to appear in Proceedings AMS, arXiv:AC/1110.1963.
\bibitem{P1} D.\ Popescu, {\em Upper bounds of depth of monomial ideals}, to appear in J. Commutative Alg., (2012),arXiv:AC/1206.3977.
\bibitem{R} A. Rauf, {\em Depth and Stanley depth of multigraded modules}, Comm.  Algebra, {\bf 38} (2010),773-784.
\bibitem{Sh} Y.H. \ Shen, {\em Lexsegment ideals of Hilbert depth 1}, (2012), arxiv:AC/1208.1822v1.
\bibitem{S} R.\ P.\ Stanley, {\em Linear Diophantine equations and local cohomology}, Invent. Math. {\bf 68} (1982) 175-193.
\bibitem{U} J.\ Uliczka, {\em Remarks on Hilbert series of graded modules over polynomial rings}, Manuscripta Math., {\bf 132} (2010), 159-168.
\bibitem{z}  A.\ Zarojanu, {\em Stanley Conjecture on three monomial primary ideals},  Bull. Math. Soc. Sc. Math. Roumanie, {\bf 55(103)},(2012), 335-338, arXiv:AC/11073211.


\end{thebibliography}
\end{document}